\documentclass[12pt,english,a4paper]{smfart}
\usepackage[T1]{fontenc}
\usepackage{lmodern}
\usepackage{smfthm}
\usepackage[headings]{fullpage}
\usepackage{amssymb}

\newcommand{\Qp}{\mathbf{Q}_p}
\newcommand{\Zp}{\mathbf{Z}_p}
\newcommand{\Cp}{\mathbf{C}_p}
\newcommand{\ZZ}{\mathbf{Z}}
\newcommand{\QQ}{\mathbf{Q}}

\newcommand{\OO}{\mathcal{O}}
\newcommand{\MM}{\mathfrak{m}}
\newcommand{\Inv}{\mathrm{U}}
\newcommand{\val}{\mathrm{val}}
\newcommand{\End}{\mathrm{End}}
\newcommand{\Hol}{\mathcal{H}}
\newcommand{\Log}{\mathrm{L}}
\newcommand{\dcroc}[1]{[\![ #1 ]\!]}
\newcommand{\wideg}{\operatorname{wideg}}

\renewcommand{\geq}{\geqslant}
\renewcommand{\leq}{\leqslant} 

\author{Laurent Berger}
\address{UMPA de l'ENS de Lyon \\
UMR 5669 du CNRS}
\email{laurent.berger@ens-lyon.fr}
\urladdr{perso.ens-lyon.fr/laurent.berger/}

\date{\today}

\title[Nonarchimedean dynamical systems and formal groups]{Nonarchimedean dynamical systems and formal groups}

\begin{document}

\begin{abstract}
We prove two theorems that confirm an observation of Lubin concerning families of $p$-adic power series that commute under composition: under certain conditions, there is a formal group such that the power series in the family are either endomorphisms of this group, or semi-conjugate to endomorphisms of this group.
\end{abstract}



\maketitle

\setlength{\baselineskip}{18pt}

\section*{Introduction}\label{intro}

Let $K$ be a finite extension of $\Qp$, and let $\OO_K$ be its ring of integers and $\MM_K$ the maximal ideal of $\OO_K$. In \cite{L94}, Lubin studied \emph{nonarchimedean dynamical systems}, namely families of elements of $X \cdot \OO_K \dcroc{X}$ that commute under composition, and remarked (page 341 of ibid.) that ``experimental evidence seems to suggest that for an invertible series to commute with a noninvertible series, there must be a formal group somehow in the background''. Various results in that direction have been obtained (by Hsia, Laubie, Li, Movahhedi, Salinier, Sarkis, Specter, ...; see for instance \cite{Li96}, \cite{Li97}, \cite{LiTAMS}, \cite{LMS02}, \cite{GS05}, \cite{GS10}, \cite{SS13}, \cite{HL},  \cite{LB16}, \cite{JS}), using either $p$-adic analysis, the theory of the field of norms or, more recently, $p$-adic Hodge theory. The purpose of this article is to prove two theorems that confirm the above observation in many new cases, using only $p$-adic analysis.

If $g(X) \in X \cdot \OO_K \dcroc{X}$, we say that $g$ is \emph{invertible} if $g'(0) \in \OO_K^\times$ and \emph{noninvertible} if $g'(0) \in \MM_K$. We say that $g$ is \emph{stable} if $g'(0)$ is neither $0$ nor a root of unity. For example, if $S$ is a formal group of finite height over $\OO_K$ and if $c \in \ZZ$ with $p \nmid c$ and $c \neq \pm 1$, then $f(X)=[p](X)$ and $u(X)=[c](X)$ are two stable power series, with $f$ noninvertible and $u$ invertible, having the following properties: the roots of $f$ and all of its iterates are simple, $f \not\equiv 0 \bmod{\MM_K}$ and $f \circ u= u \circ f$. Our first result is a partial  converse of this. If $f(X) \in X \cdot \OO_K \dcroc{X}$, let $\Inv_f$ denote the set of invertible power series $u(X) \in X \cdot \OO_K \dcroc{X}$ such that $f \circ u= u \circ f$, and let $\Inv_f'(0) = \{ u'(0), \ u \in \Inv_f\}$. This is a subgroup of $\OO_K^\times$.

\begin{enonce*}{Theorem A}
Let $K$ be a finite extension of $\Qp$ such that $e(K/\Qp) \leq p-1$, and let $f(X) \in X \cdot \OO_K \dcroc{X}$ be a noninvertible stable series. Suppose that 
\begin{enumerate}
\item the roots of $f$ and all of its iterates are simple, and $f \not\equiv 0 \bmod{\MM_K}$;
\item there is a subfield $F$ of $K$ such that $f'(0) \in \MM_F$ and such that $\Inv_f'(0) \cap \OO_F^\times$ is an open subgroup of $\OO_F^\times$.
\end{enumerate}
Then there is a formal group $S$ over $\OO_K$ such that $f \in \End(S)$ and $\Inv_f \subset \End(S)$.
\end{enonce*}

Condition (1) can be checked using the following criterion (proposition \ref{critsimp}).

\begin{enonce*}{Criterion A}
If $f(X) \in X \cdot \OO_K \dcroc{X}$ is a noninvertible stable series with $f \not\equiv 0 \bmod{\MM_K}$, and if $f$ commutes with a stable invertible series $u(X) \in X \cdot \OO_K \dcroc{X}$, then the roots of $f$ and all of its iterates are simple if and only if $f'(X)/f'(0) \in 1 + X \cdot \OO_K \dcroc{X}$. 
\end{enonce*}

If $K=\Qp$, condition (2) of Theorem A amounts to requiring the existence of a stable invertible series that commutes with $f$.

\begin{enonce*}{Corollary A}
If $f(X) \in X \cdot \Zp \dcroc{X}$ is a noninvertible stable series such that the roots of $f$ and  all of its iterates are simple and $f \not\equiv 0 \bmod{p}$, and if $f$ commutes with a stable invertible series $u(X) \in X \cdot \Zp \dcroc{X}$, then there is a formal group $S$ over $\Zp$ such that $f \in \End(S)$ and $\Inv_f \subset \End(S)$.
\end{enonce*}

There are examples of commuting power series where $f$ does not have simple roots, for instance $f(X)=9X+6X^2+X^3$ and $u(X)=4X+X^2$ with $K=\QQ_3$ (more examples can be constructed following the discussion on page 344 of \cite{L94}). It seems reasonable to expect that if $f$ and $u$ are two stable noninvertible and invertible power series that commute, with $f \not\equiv 0 \bmod{\MM_K}$, then there exists a formal group $S$, two endomorphisms $f_S$ and $u_S$ of $S$, and a nonzero power series $h$ such that $f \circ h = h \circ f_S$ and $u \circ h = h \circ u_S$. We then say that $f$ and $f_S$ are \emph{semi-conjugate}, and $h$ is an \emph{isogeny} from $f_S$ to $f$ (see for instance \cite{Li97}). 

The simplest case where this occurs is when $m$ is an integer $\geq 2$, and the nonzero roots of $f$ and all of its iterates are of multiplicity $m$ (for an example of a more complicated case, see remark \ref{chebytwo}). In this simplest case, we have the following.

\begin{enonce*}{Theorem B}
Let $K$ be a finite extension of $\Qp$, let $f(X) \in X \cdot \OO_K \dcroc{X}$ be a noninvertible stable series and take $m \geq 2$. Let $h(X)=X^m$. Suppose that 
\begin{enumerate}
\item the nonzero roots of $f$ and all of its iterates are of multiplicity $m$ 
\item $f \not\equiv 0 \bmod{\MM_K}$.
\end{enumerate}

Then there exists a finite unramified extension $L$ of $K$ and a noninvertible stable series $f_0(X) \in X \cdot \OO_L \dcroc{X}$ with $f_0 \not\equiv 0 \bmod{\MM_K}$, such that $f \circ h = h \circ f_0$, and the roots of $f_0$ and all of its iterates are simple.

If in addition $u$ is an element of $\Inv_f$ with $u'(0) \equiv 1 \bmod{\MM_K}$, then there exists $u_0 \in \Inv_{f_0}$ such that $u \circ h = h \circ u_0$. Finally, if there is a subfield $F$ of $K$ such that $f'(0) \in \MM_F$ and such that $\Inv_f'(0) \cap \OO_F^\times$ is an open subgroup of $\OO_F^\times$, then $(f_0^{\circ m})'(0) \in \MM_F$ and $\Inv_{f_0}'(0) \cap \OO_F^\times$ is an open subgroup of $\OO_F^\times$.
\end{enonce*}

Condition (1) can be checked using the following criterion (proposition \ref{critmultm}).

\begin{enonce*}{Criterion B}
If $f(X) \in X \cdot \OO_K \dcroc{X}$ is a noninvertible stable series with $f \not\equiv 0 \bmod{\MM_K}$, and if $f$ commutes with a stable invertible series $u(X) \in X \cdot \OO_K \dcroc{X}$, then the nonzero roots of $f$ and all of its iterates are of multiplicity $m$ if and only if the nonzero roots of $f$ are of multiplicity $m$, and the set of roots of $f'$ is included in the set of roots of $f$.
\end{enonce*}

We have the following simple corollary of Theorem B when $K=\Qp$.

\begin{enonce*}{Corollary B}
If $m \geq 2$ and $f(X) \in X \cdot \Zp \dcroc{X}$ is a noninvertible stable series such that the nonzero roots of $f$ and all of its iterates are of multiplicity $m$ and $f \not\equiv 0 \bmod{p}$, and if $f$ commutes with a stable invertible series $u(X) \in X \cdot \Zp \dcroc{X}$, then there is an unramified extension $L$ of $\Qp$, a formal group $S$ over $\OO_L$ and $f_S \in \End(S)$ such that $f \circ X^m = X^m \circ f_S$.
\end{enonce*}

Theorem A implies conjecture 5.3 of \cite{HL} for those $K$ such that $e(K/\Qp) \leq p-1$. It also provides a new simple proof (that does not use $p$-adic Hodge theory) of the main theorem of \cite{JS}. Note also that Theorem A holds without the restriction ``$e(K/\Qp) \leq p-1$'' if $f'(0)$ is a uniformizer of $\OO_K$ (see \cite{JS17}). This implies ``Lubin's conjecture'' formulated at the very end of \cite{GS10} (this conjecture is proved in \cite{LB16} using $p$-adic Hodge theory, when $K$ is a finite Galois extension of $\Qp$) as well as ``Lubin's conjecture'' on page 131 of \cite{GS05} over $\Qp$ if $f \not\equiv 0 \bmod{p}$.

The results of \cite{HL}, \cite{LB16} and \cite{JS} are proved under strong additional assumptions on $\wideg(f)$ (namely that $\wideg(f)=p$ in \cite{JS}, or that $\wideg(f)=p^h$, where $h$ is the residual degree of $K$, in \cite{HL} and \cite{LB16}). Theorem A is the first general result in this direction that makes no assumption on $\wideg(f)$, besides assuming that it is finite. It also does not assume that $f'(0)$ is a uniformizer of $\OO_K$.

Theorem A and its corollary are proved in section \S\ref{proof} and theorem B and its corollary are proved in section \S\ref{condens}. 

\section{Nonarchimedean dynamical systems}
\label{serlub}

Whenever we talk about the roots of a power series, we mean its roots in the $p$-adic open unit disk $\MM_{\Cp}$. Recall that the \emph{Weierstrass degree} $\wideg(g(X))$ of a series $g(X) = \sum_{i \geq1} g_i X^i \in X \cdot \OO_K \dcroc{X}$ is the smallest integer $i \leq +\infty$ such that $g_i \in \OO_K^\times$. We have $\wideg(g) = +\infty$ if and only if $g \equiv 0 \bmod{\MM_K}$.

If $r<1$, let $\Hol(r)$ denote the set of power series in $K\dcroc{X}$ that converge on the closed disk $\{ z \in \MM_{\Cp}$ such that $|z|_p \leq r\}$. If $h \in \Hol(r)$, let $\| h \|_r = \sup_{|z|_p \leq r} |h(z)|_p$. The space $\Hol(r)$ is complete for the norm $\|{\cdot}\|_r$. Let $\Hol = \projlim_{r<1} \Hol(r)$ be the ring of holomorphic functions on the open unit disk.

Throughout this article, $f(X) \in X \cdot \OO_K \dcroc{X}$ is a stable noninvertible series such that $\wideg(f) < + \infty$, and $\Inv_f$ denotes the set of invertible power series $u(X) \in X \cdot \OO_K \dcroc{X}$ such that $f \circ u= u \circ f$. 

\begin{lemm}
\label{lubcom}
A series $g(X) \in X \cdot K \dcroc{X}$ that commutes with $f$ is determined by $g'(0)$.
\end{lemm}

\begin{proof}
This is proposition 1.1 of \cite{L94}.
\end{proof}

\begin{prop}
\label{wideg}
If $\Inv_f$ contains a stable invertible series, then there exists a power series $g(X) \in X \cdot \OO_K \dcroc{X}$ and an integer $d \geq 1$ such that $f(X) \equiv g(X^{p^d}) \bmod{\MM_K}$. 

We have $\wideg(f) = p^d$ for some $d \geq 1$.
\end{prop}

\begin{proof}
This is the main result of \cite{L94}. See (the proof of) theorem 6.3 and corollary 6.2.1 of ibid.
\end{proof}

\begin{prop}
\label{lublog}
There is a (unique) power series $\Log(X) \in X + X^2 \cdot K \dcroc{X}$ such that $\Log \circ f = f'(0) \cdot \Log$ and $\Log \circ u = u'(0) \cdot \Log$ if $u \in \Inv_f$. The series $\Log(X)$ converges on the open unit disk, and $\Log(X) = \lim_{n \to +\infty} f^{\circ n}(X) / f'(0)^n$ in the Fr\'echet space $\Hol$.
\end{prop}

\begin{proof}
See propositions 1.2, 1.3 and 2.2 of \cite{L94}.
\end{proof}

\begin{lemm}
\label{primroot}
If $f(X) \in X \cdot \OO_K \dcroc{X}$ is a noninvertible stable series and if $f$ commutes with a stable invertible series $u$, then every root of $f'$ is a root of $f^{\circ n}$ for some $n \gg 0$.
\end{lemm}

\begin{proof}
This is corollary 3.2.1 of \cite{L94}.
\end{proof}

\begin{prop}
\label{critsimp}
If $f(X) \in X \cdot \OO_K \dcroc{X}$ is a noninvertible stable series with $f \not\equiv 0 \bmod{\MM_K}$, and if $f$ commutes with a stable invertible series $u$, then the roots of $f$ and all of its iterates are simple if and only if $f'(X)/f'(0) \in 1 + X \cdot \OO_K \dcroc{X}$. 
\end{prop}

\begin{proof}
We have $(f^{\circ n})'(X) = f'(f^{\circ n-1}(X)) \cdots f'(f(X)) \cdot f'(X)$. If $f'(X)/f'(0) \in 1 + X \cdot \OO_K \dcroc{X}$, then the derivative of $f^{\circ n}(X)$ belongs to  $f'(0)^n \cdot(1 + X \cdot \OO_K \dcroc{X})$ and hence has no roots. The roots of $f^{\circ n}(X)$ are therefore simple.

By lemma \ref{primroot}, any root of $f'(X)$ is also a root of $f^{\circ n}$ for some $n \gg 0$. If the roots of $f^{\circ n}(X)$ are simple for all $n \geq 1$, then $f'(X)$ cannot have any root, and hence $f'(X)/f'(0) \in 1+X \OO_K \dcroc{X}$.
\end{proof}

\section{Formal groups}
\label{proof}

We now prove theorem A. Let $S(X,Y) = \Log^{\circ -1}(\Log(X)+\Log(Y)) \in K \dcroc{X,Y}$. By proposition \ref{lublog}, $S$ is a formal group law over $K$ such that $f$ and all $u \in \Inv_f$ are endomorphisms of $S$. In order to prove theorem A, we show that $S(X,Y) \in \OO_K \dcroc{X,Y}$. Write $S(X,Y) = \sum_{j \geq 0} s_j(X) Y^j$. 

\begin{lemm}
\label{lifact}
If $\Log'(X) \in \OO_K \dcroc{X}$, then $s_j(X) \in j!^{-1} \cdot \OO_K \dcroc{X}$ for all $j \geq 0$.
\end{lemm}

\begin{proof}
This is lemma 3.2 of \cite{Li96}.
\end{proof}

\begin{lemm}
\label{logprim}
If the roots of $f^{\circ n}(X)$ are simple for all $n \geq 1$, then $\Log'(X) \in \OO_K \dcroc{X}$.
\end{lemm}

\begin{proof}
This is sketched in the proof of theorem 3.6 of \cite{Li96}. We give a complete argument for the convenience of the reader. 

We have $(f^{\circ n})'(X) = f'(f^{\circ n-1}(X)) \cdots f'(f(X)) \cdot f'(X)$, and by proposition \ref{critsimp},  $f'(X)/f'(0) \in 1+X \OO_K \dcroc{X}$. We have $\Log(X) = \lim_{n \to +\infty} f^{\circ n}(X) / f'(0)^n$ by proposition \ref{lublog},  so that
\[ \Log'(X) = \lim_{n \to +\infty} \frac{(f^{\circ n})'(X)} {f'(0)^n} = \lim_{n \to +\infty} \frac{ f'(f^{\circ n-1}(X))} {f'(0)} \cdots \frac{f'(f(X))} {f'(0)} \cdot \frac{f'(X)} {f'(0)}, \]  
and hence $\Log'(X) \in 1+X \OO_K \dcroc{X}$.
\end{proof}

\begin{theo}
\label{siint}
If $e(K/\Qp) \leq p-1$, then $s_j(X) \in \OO_K \dcroc{X}$ for all $j \geq 0$.
\end{theo}

\begin{proof}
For all $n \geq 1$, the power series $u_n(X) = S(X,f^{\circ n}(X))$ belongs to $X \cdot K \dcroc{X}$ and satisfies $u_n \circ f = f \circ u_n$. Since $\Inv_f'(0) \cap \OO_F^\times$ is an open subgroup of $\OO_F^\times$, there exists $n_0$ such that if $n \geq n_0$, then $u_n'(0) = 1+f'(0)^n \in  \Inv_f'(0)$. We then have $u_n \in \Inv_f$ by lemma \ref{lubcom}.

In order to prove the theorem, we therefore prove that if $S(X,f^{\circ n}(X)) \in \OO_K \dcroc{X}$ for all $n \geq n_0$, then $s_i(X) \in \OO_K \dcroc{X}$ for all $i \geq 0$. If $j \geq 1$, let 
\[ a_j(X) = f^{\circ n}(X) \sum_{i \geq 0} s_{j+i}(X) f^{\circ n}(X)^i = s_j(X) f^{\circ n}(X) + s_{j+1}(X) f^{\circ n}(X)^2 + \cdots. \]
We prove by induction on $j$ that $s_0(X),\hdots,s_{j-1}(X)$ as well as $a_j(X)$ belong to $\OO_K \dcroc{X}$. This holds for $j =1$; suppose that it holds for $j$.

We claim that if $h \in \Hol(r)$ and $\| h \|_r < p^{-1/(p-1)}$, then  $\sum_{i \geq 0} s_{j+i}(X) h(X)^i$ converges in $\Hol(r)$. Indeed, if $s_p(j+i)$ denotes the sum of the digits of $j+i$ in base $p$, then 
\[ \val_p((j+i)!) = \frac{j+i-s_p(j+i) }{p-1} \leq \frac{i}{p-1} + \frac{j}{p-1}. \]

Let $\pi$ be a uniformizer of $\OO_K$ and let $e=e(K/\Qp)$ so that $|\pi|_p=p^{-1/e}$. By proposition \ref{wideg}, we have \[ f^{\circ n}(X) \in \pi X \cdot \OO_K \dcroc{X} + X^{q^n} \cdot \OO_K \dcroc{X^{q^n}}, \] 
where $q=p^d=\wideg(f)$, so that $\| f^{\circ n}(X) \|_r \leq \max( rp^{-1/e}, r^{q^n} )$. If $\rho_n=p^{-1/(e(q^n-1))}$, then \[ \| f^{\circ n}(X) \|_{\rho_n} \leq p^{-q^n/(e(q^n-1))} < p^{-1/e} \leq p^{-1/(p-1)} \] and the series $\sum_{i \geq 0} s_{j+i}(X) f^{\circ n}(X)^i$ therefore converges in $\Hol(\rho_n)$.

We have $f^{\circ n}(X) \in \pi X \cdot \OO_K \dcroc{X} + X^{q^n} \cdot \OO_K \dcroc{X^{q^n}}$, as well as $\wideg(f^{\circ n})=q^n$. By the theory of Newton polygons, all the zeroes $z$ of $f^{\circ n}(X)$ satisfy $\val_p(z) \geq 1/(e(q^n-1))$, and hence $|z|_p \leq \rho_n$. The equation $a_j(X) = f^{\circ n}(X) \sum_{i \geq 0} s_{j+i}(X) f^{\circ n}(X)^i$ holds in $\Hol(\rho_n)$, and this implies that $a_j(z)=0$ for all $z$ such that $f^{\circ n}(z) = 0$. Since all the zeroes of $f^{\circ n}(X)$ are simple and $f^{\circ n}(X) \not\equiv 0 \bmod{\pi}$, the Weierstrass preparation theorem implies that $f^{\circ n}(X)$ divides $a_j(X)$ in $\OO_K \dcroc{X}$, and hence that \[ s_j(X) + s_{j+1}(X) f^{\circ n}(X) + s_{j+2}(X) f^{\circ n}(X)^2 + \cdots \in \OO_K \dcroc{X}. \] 

Choose some $0 < \rho <1$ and take $n \geq n_0$ such that $\rho_n \geq \rho$. We have 
\[ f^{\circ n}(X) = f(f^{\circ n-1}(X)) \in \pi f^{\circ n-1}(X) \cdot \OO_K\dcroc{X} + f^{\circ n-1}(X)^q \cdot \OO_K \dcroc{X}. \] Therefore $\|f^{\circ n}(X)\|_\rho \to 0$ as $n \to +\infty$, and $\| s_{j+1}(X) f^{\circ n}(X) + s_{j+2}(X) f^{\circ n}(X)^2 + \cdots \|_\rho \to 0$ as $n \to +\infty$. The series $s_j(X)$ is therefore in the closure of $\OO_K \dcroc{X}$ inside $\Hol(\rho)$ for $\|{\cdot}\|_\rho$, which is $\OO_K \dcroc{X}$.

This proves that $s_j(X)$ as well as $s_{j+1}(X) f^{\circ n}(X) + s_{j+2}(X) f^{\circ n}(X)^2 + \cdots$ belong to $\OO_K \dcroc{X}$. This finishes the induction and hence the proof of the theorem.
\end{proof}

Theorem A now follows: $S$ is a formal group over $\OO_K$ such that $f \in \End(S)$. Any power series $u(X) \in X \cdot \OO_K \dcroc{X}$ that commutes with $f$ also belongs to $\End(S)$, since $u(X) = [u'(0)](X)$ by lemma \ref{lubcom}. In particular, $\Inv_f \subset \End(S)$.

To prove corollary A, note that we can replace $u$ by $u^{\circ p-1}$ and therefore assume that $u'(0) \in 1+p\Zp$. In this case, $u^{\circ m}$ is defined for all $m \in \Zp$ by proposition 4.1 of \cite{L94} and $\Inv_f'(0)$ is therefore an open subgroup of $\Zp^\times$.

\section{Semi-conjugation}
\label{condens}

We now prove theorem B. Assume therefore that the nonzero roots of $f$ and all of its iterates are of multiplicity $m$. Let $h(X)=X^m$.

Since $q = \wideg(f)$ is finite, we can write $f(X)=X \cdot g(X) \cdot v(X)$ where $g(X) \in \OO_K[X]$ is a distinguished polynomial and $v(X)\in \OO_K \dcroc{X}$ is a unit. If the roots of $g(X)$ are of multiplicity $m$, then $g(X)=g_0(X)^m$ for some $g_0(X) \in \OO_K[X]$. Write $v(X) = [c] \cdot (1+ w(X))$ where $c \in k_K$ (and $[c]$ is its Teichm\"uller lift) and $w(X) \in (\MM_K,X)$. Since $m \cdot \deg(g) = q-1$, $m$ is prime to $p$ and there exists a unique $w_0(X) \in (\MM_K,X)$ such that $1+w(X) = (1+w_0(X))^m$. If $f_0(X) = [c^{1/m}] \cdot X \cdot g_0(X^m) \cdot (1+w_0(X^m))$, then
\[ f \circ h(X) = f(X^m) = [c] \cdot X^m \cdot g_0(X^m)^m \cdot (1+w_0(X^m))^m = f_0(X)^m = h \circ f_0(X). \]
It is clear that $f_0 \not\equiv 0 \bmod{\MM_K}$. If we write $f^{\circ n}(X) = X \cdot \prod_\alpha (X-\alpha)^m \cdot v_n(X)$ with $v_n$ a unit of $\OO_K\dcroc{X}$, and where $\alpha$ runs through the nonzero roots of $f^{\circ n}$, then \[ f^{\circ n}(X^m) = X^m \cdot \prod_\alpha (X^m-\alpha)^m \cdot v_n(X^m), \] so that all the roots of $f^{\circ n}(X^m)$ have multiplicity $m$. Since $f^{\circ n}(X^m)=f_0^{\circ n}(X)^m$, the roots of $f_0$ and all of its iterates are simple. This finishes the proof of the first part of the theorem, with $L=K([c^{1/m}])$. 

If $u \in \Inv_f$ and $u'(0) \in 1+\MM_K$, then there is a unique $u_0(X) \in 1+(\MM_K,X)$ such that $u_0(X)^m = u(X^m)$. We have $u_0'(0) = u'(0)^{1/m}$ and $(f_0 \circ u_0)^m = (u_0 \circ f_0)^m$ as well as $(f_0 \circ u_0)'(0) = (u_0 \circ f_0)'(0)$, so that $u_0 \in \Inv_{f_0}$. This proves the existence of $u_0$. Since $f(X^m)=f_0(X)^m$, we have $f'(0) = f_0'(0)^m$. We then have $(f_0^{\circ m})'(0) = f_0'(0)^m = f'(0) \in \MM_F$. This finishes the proof of the last claim of theorem B.

Corollary B follows from theorem B in the same way that corollary A followed from theorem A.

\begin{exem}
\label{cheby}
If $p=3$ and $f(X)=9X+6X^2+X^3$ and $u(X)=4X+X^2$, so that $f \circ u = u  \circ f$, then $f(X)=X(X+3)^2$ and $f'(X)=3(X+3)(X+1)$. The nonzero roots of $f$ and all of its iterates are therefore of multiplicity $2$. We have $f(X^2) = (X(X^2+3))^2$ so that $f_0(X)=3X+X^3$, and the corresponding formal group is $\mathbf{G}_m$ (this is a special case of the construction given on page 344 of \cite{L94}).
\end{exem}

\begin{prop}
\label{critmultm}
If $f(X) \in X \cdot \OO_K \dcroc{X}$ is a noninvertible stable series with $f \not\equiv 0 \bmod{\MM_K}$, and if $f$ commutes with a stable invertible series $u(X) \in X \cdot \OO_K \dcroc{X}$, then the nonzero roots of $f$ and all of its iterates are of multiplicity $m$ if and only if the nonzero roots of $f$ are of multiplicity $m$ and the set of roots of $f'$ is included in the set of roots of $f$.
\end{prop}

\begin{proof}
If the nonzero roots of $f$ and all of its iterates are of multiplicity $m$, then the nonzero roots of $f$ are of multiplicity $m$. Hence if $\alpha$ is a root of $f^{\circ n}(X)$ with $f(\alpha) \neq 0$,  the equation $f(X) = f(\alpha)$ has simple roots. Since $\alpha$ is one of these roots, we have  $f'(\alpha) \neq 0$.  By lemma \ref{primroot}, any root of $f'(X)$ is also a root of $f^{\circ n}$ for some $n \geq 1$. This implies that the set of roots of $f'$ is included in the set of roots of $f$.

Conversely, suppose that the nonzero roots of $f$ are of multiplicity $m$, and that $f'(\beta) \neq 0$ for any $\beta$ that is not a root of $f$. If $\alpha$ is a nonzero root of $f^{\circ n}$ for some $n \geq 1$, then this implies that the equation $f(X) = \alpha$ has simple roots, so that the nonzero roots of $f$ and all of its iterates are of multiplicity $m$.
\end{proof}

\begin{rema}
\label{chebytwo}
If $p=2$ and $f(X)=4X+X^2$ and $u(X)=9X+6X^2+X^3$, then $f \circ u = u  \circ f$. The roots $0$ and $-4$ of $f$ are simple, but $f^{\circ 2}(X)=X(X+4)(X+2)^2$ has a double root. In this case, $f$ is still semi-conjugate to an endomorphism of $\mathbf{G}_m$, but via the more complicated map $h(X)=X^2/(1+X)$ (see the discussion after corollary 3.2.1 of \cite{L94}, and example 2 of \cite{Li96}).
\end{rema}

\providecommand{\bysame}{\leavevmode ---\ }
\providecommand{\og}{``}
\providecommand{\fg}{''}
\providecommand{\smfandname}{\&}
\providecommand{\smfedsname}{\'eds.}
\providecommand{\smfedname}{\'ed.}
\providecommand{\smfmastersthesisname}{M\'emoire}
\providecommand{\smfphdthesisname}{Th\`ese}


\begin{thebibliography}{LMS02}

\bibitem[Ber17]{LB16}
{\scshape L.~Berger} -- {\og Lubin's conjecture for full {$p$}-adic dynamical
  systems\fg}, in \emph{Publications math\'ematiques de {B}esan\c{c}on.
  {A}lg\`ebre et th\'eorie des nombres, 2016}, Publ. Math. Besan\c{c}on
  Alg\`ebre Th\'eorie Nr., vol. 2016, Presses Univ. Franche-Comt\'e,
  Besan\c{c}on, 2017, p.~19--24.

\bibitem[HL16]{HL}
{\scshape L.-C. Hsia {\normalfont \smfandname} H.-C. Li} -- {\og Ramification
  filtrations of certain abelian {L}ie extensions of local fields\fg}, \emph{J.
  Number Theory} \textbf{168} (2016), p.~135--153.

\bibitem[Li96]{Li96}
{\scshape H.-C. Li} -- {\og When is a {$p$}-adic power series an endomorphism
  of a formal group?\fg}, \emph{Proc. Amer. Math. Soc.} \textbf{124} (1996),
  no.~8, p.~2325--2329.

\bibitem[Li97a]{Li97}
\bysame , {\og Isogenies between dynamics of formal groups\fg}, \emph{J. Number
  Theory} \textbf{62} (1997), no.~2, p.~284--297.

\bibitem[Li97b]{LiTAMS}
\bysame , {\og {$p$}-adic power series which commute under composition\fg},
  \emph{Trans. Amer. Math. Soc.} \textbf{349} (1997), no.~4, p.~1437--1446.

\bibitem[LMS02]{LMS02}
{\scshape F.~Laubie, A.~Movahhedi {\normalfont \smfandname} A.~Salinier} --
  {\og Syst\`emes dynamiques non archim\'ediens et corps des normes\fg},
  \emph{Compositio Math.} \textbf{132} (2002), no.~1, p.~57--98.

\bibitem[Lub94]{L94}
{\scshape J.~Lubin} -- {\og Nonarchimedean dynamical systems\fg},
  \emph{Compositio Math.} \textbf{94} (1994), no.~3, p.~321--346.

\bibitem[Sar05]{GS05}
{\scshape G.~Sarkis} -- {\og On lifting commutative dynamical systems\fg},
  \emph{J. Algebra} \textbf{293} (2005), no.~1, p.~130--154.

\bibitem[Sar10]{GS10}
\bysame , {\og Height-one commuting power series over {$\Bbb Z_p$}\fg},
  \emph{Bull. Lond. Math. Soc.} \textbf{42} (2010), no.~3, p.~381--387.

\bibitem[Spe17]{JS17}
{\scshape J.~Specter} -- personal communication, 2017.

\bibitem[Spe18]{JS}
\bysame , {\og The crystalline period of a height one {$p$}-adic dynamical
  system\fg}, \emph{Trans. Amer. Math. Soc.} \textbf{370} (2018), no.~5,
  p.~3591--3608.

\bibitem[SS13]{SS13}
{\scshape G.~Sarkis {\normalfont \smfandname} J.~Specter} -- {\og Galois
  extensions of height-one commuting dynamical systems\fg}, \emph{J. Th\'eor.
  Nombres Bordeaux} \textbf{25} (2013), no.~1, p.~163--178.

\end{thebibliography}
\end{document}